\documentstyle[12pt,twoside,leqno]{amsart}

\title {The equality case in Wu--Yau inequality}
\author{Stefano Trapani}
\address{Stefano Trapani \\ Università\`a di Roma \lq\lq Tor Vergata\rq\rq{} \\ Via della Ricerca Scientifica 1 \\ I-00179 Roma.}
\email{trapani at mat.uniroma2.it}
\date  {\today} 

\newtheorem{theorem}             {Theorem}    [section]

\newtheorem{remark}     [theorem]{Remark}      %
\newenvironment{proof}{\begin{trivlist}\item[]{\em Proof.}}%
                      {\rule{.1in}{.1in}\end{trivlist}}
\newcounter    {assertcount}

\renewcommand{\epsilon}{\varepsilon}
\renewcommand{\phi}    {\varphi}

\begin{document}

\maketitle

\begin{abstract}
In recent papers Wu-Yau, Tosatti-Yang and Diverio-Trapani, used some natural differential inequalities for compact K\"ahler manifolds with quasi negative holomorphic sectional curvature to derive positivity of the canonical bundle. In this note we study the equality case of these inequalities.

\end{abstract}

\section{Introduction}
Let $(X, \omega) $ be a compact connected K\"ahler manifold of complex dimension $n.$ Let  ${\bf k}_x(v) $ be the holomorphic sectional curvature of $\omega$ in  $x$ and in the direction of $ v \in T_x(X)$ with $||v||_{\omega} = 1.$  That is the sectional curvature of $\omega$ on the real two dimensional subspace of $T_x(X)$ spanned by $v,Jv$ where $J$ is the complex structure of $X.$   
Assume that such holomorphic sectional curvature is "quasi negative" i.e. that ${\bf k}_x(v) \leq 0 $ for all $x \in X$ and $v$ unit vector in $T_x(X),$ and that there exists a point $x_0 \in X$ such that 
${ \bf k}_{x_0}(v) < 0$ for every  unit vector $v \in T_{x_0}(X).$ 
It was conjectured by S. T. Yau that such manifolds have ample canonical bundle. This  conjecture has been proved by Yau himself together with D. Wu in \cite{W-Y} when the manifold is projective and the curvature is strictly negative, shortly after  this result  was extended by Tosatti and Yang  in \cite{T-Y} to the case where the curvature is strictly negative and the manifold is K\"ahler. Finally the general conjecture was proved in \cite{D-T},  later all these results were proved in a unified way in   \cite{W-Y-2}. The method of proof in all these papers is similar and it stems from the original \cite{W-Y} where the main point is the proof of a certain differential inequality valid for manifolds with semi negative holomorphic sectional curvature. In \cite{D-T} the main result is also deduced from a differential inequality. In another direction there is a long history of uniformization theorems derived from differential geometrical conditions, in this note we show that we have  equality  in the inequalities obtained in \cite{W-Y} and \cite{D-T}  if and only if the manifold is covered by the ball and the metric $\omega$ is K\"ahler Einstein.

\section{The  inequalities }
Let $(X, \omega) $ be a compact connected K\"ahler manifold of complex dimension $n$ with quasi negative holomorphic sectional curvature.  
Let $H(x) = max_{(v \in T_x(X)   :  |v|_{\omega} = 1)} { \bf k}_x(v)$
\[ M(x) = -\left( \frac{n+1}{2n} H(x) \right)  . \] So by assumption $M \geq 0 $ and $M(x_0) > 0.$ By \cite{W-Y}, \cite{T-Y}, \cite{D-T}, \cite{W-Y-2}  the canonical bundle $K_X$ is ample. It follows from \cite{Y} that there exist a K\"ahler Einstein metric $\omega'$ such that 
$Ricci(\omega') = - \omega'.$ 
 
Let $\overline{M} = \frac{ \int_X M(x) {\omega'}^n }{  \int_X {\omega'}^n} $ be the mean value of $M$ with respect to $\omega'.$  Note that
$\overline{M} > 0.$ 
Let $S = \frac{n \omega \wedge { \omega'}^{n-1}}{ { \omega'}^n}$ be the trace of $\omega$ with respect to $\omega'.$ 
Let $f$ be the  unique solution of the equation 
$ \Delta_{\omega'} (f) = \frac{M}{\overline{M}} - 1, \  min_X f = 0. $
We consider the differential equation 
\begin{equation} \Delta_{\omega'} (h) = M exp(h) -1 \label{eq1}   \end{equation} 
It is  proved in \cite{W-Y} that the function $log(S)$ is a subsolution of equation (\ref{eq1}) i. e.

\begin{equation} \Delta_{\omega'} (log(S)) \geq M exp( log(S)) -1 \label{ineq2}   \end{equation}

Moreover it is proved in   \cite{D-T} that the function  
$ f - log(\overline{M})$ is a super solution of the same equation: i.e.  
\begin{equation} \Delta_{\omega'} (f - log(\overline{M})  \leq  M exp( f - log(\overline{M})) -1. \label{ineq3}   \end{equation} 
See also \cite{K-W}.

Using (\ref{ineq2}) and (\ref{ineq3}) we can derive the inequality  

\begin{equation}   log(S) \leq f - log({\overline{M}}) \label{ineq4} \end{equation} 
see \cite{D-T}.

\section{The equalities} 

Here we want to study under which conditions the above inequalities are equalities.

First of all we have 

\begin{theorem} 
Let $(X,\omega)$ be a compact connected K\"ahler manifold of complex dimension $n,$ such that the holomorphic sectional curvature of $\omega$ is quasi negative then:

\[ \begin{array}{l} \mbox{i)    If $ n =1$ then the function $\log(S)$ solves equation (\ref{eq1})  and $X$ is covered by the disk,} \\ 
\mbox{moreover the function $ f - log(\overline{M})$ solves equation (\ref{eq1})} \\ \mbox{ if and only if $\omega = \lambda \omega'$  for some positive constant $\lambda$.} \\ \\
\mbox{ii)   If $n \geq 2$ then then the function $\log(S)$ solves equation ( \ref{eq1})} \\ \mbox{ if and only if  $X$ is covered by the ball and $\omega = \lambda \omega'$  for some positive constant $\lambda$.}  \end{array} \] \label{dim1} \end{theorem} 

\begin{proof}
Assume  $n =1$ then $X$ is covered by the disk by Gauss Bonnet theorem, also we can see that $\log(S)$ solves equation (\ref{eq1}) by direct calculation. Moreover $f - log(\overline{M})$ solves equation (\ref{eq1}) if and only if $\frac{M}{\overline{M}} -1 = \frac{M}{\overline{M}}exp(f) -1$ if and only if $\frac{M}{\overline{M}}(exp(f) -1) \equiv 0.$ Let $\Omega$ be the non empty open set where $M$ is non zero. We then have that $f \equiv 0$ in $\Omega $ hence $  M \equiv \overline{M}$ on $\Omega$ and $M \equiv 0$ on $X \setminus \Omega.$ Since $X$ is connected, it follows that   
$f - log(\overline{M})$ solves equation (\ref{eq1}) if and only if the function $M$ is a positive constant. However in dimension $1$ the function $- M$ is the Gaussian curvature, claim i) follows.   
We now prove ii).  If $X$ is covered by the ball and $\omega= \lambda \omega'$ with $\lambda $ constant, the metric $\omega'$ is induced by the Bergman metric on the ball, see Example 6.6 chapter  IX pag.163 \cite{K-N}, hence   it has constant negative holomorphic sectional curvature ${\bf k'} = \frac{-2}{n+1},$ see \cite{K-N} chapter IX remark pag.168 and Theorem 7.8 (2) pag. 169. Therefore the metric $ \omega$ has constant holomorphic sectional curvature  ${\bf k} = \frac{-2}{ \lambda(n+1)}.$ Hence the function $M =  \frac{1}{ n \lambda}$ is a positive constant, so  $M = \overline{M} $ and $log(S) = log(n \lambda)$ is also constant. Therefore $log(S)$ solves equation ( \ref{eq1}). Assume vice versa that $n \geq 2$ and that $\log(S)$ solves equation ( \ref{eq1}), to prove the claim we  
look more closely at the proof of proposition 9 in \cite{W-Y}.
For every $x \in X$ let us fix holomorphic local coordinate centered at $x,$ such that the local expression $g_{i,\bar{j}} $ of the metric $\omega$ satisfies $g_{i,\bar{j}}(0)  = \delta_{i,j} $ and 
$ \frac{\partial g_{i,\bar{j}}}{\partial z_k}(0) =  \frac{\partial g_{i,\bar{j}}}{\partial { \bar{z}}_k}(0) = 0,$ whereas the local expression 
${g'}_{i,\bar{j}} $ of the metric $\omega'$ is diagonal at $0.$  
We denote by 
$R_{i, \bar{j}.k.\bar{l}}$ and  ${R'}_{i, \bar{j}.k.\bar{l}}$ the component of the curvature tensor of $\omega$ and $\omega'$ respectively in these coordinate system. We also denote by $r_{i \bar{j}},$ and ${r'}_{i\bar{j}},$ the components of the Ricci tensor of $\omega$ and $\omega'$ respectively. We have the following formula, see 
\cite{Y-2} and \cite{W-Y},

$\Delta_{\omega'} (S) = I_1 + I_2 + I_3$ where 
\[ I_1 = \sum_{j=1}^n \frac{{r'}_{i,\bar{j}}}{({g'}_{j,\bar{j}})^2} \] 
\[ I_2 =  \sum_{j,k,l = 1}^n \frac{| \frac{ \partial {g'}_{j,\bar{k}}}{\partial z_l}|^2}{{g'}_{j,\bar{j}} ({g'}_{k,\bar{k}})^2 {g'}_{l,\bar{l}} } \]
\[ I_3 = -\sum_{j,k=1}^n \frac{R_{j, \bar{j},k, \bar{k}}}{{g'}_{j,\bar{j}} {g'}_{k,\bar{k}} }. \]
On the other hand 
\[ \Delta'( \log(S)) = \frac{\Delta'(S)}{S} -  \frac{| \nabla'(S)|^2}{S^2 } \]
Moreover in \cite{W-Y} Wu and Yau proved the inequalities:
\[ I_1 \geq -S \]
\[ I_2 \geq \frac{| \nabla'(S)|^2}{S} \]
\[ I_3 \leq M S^2. \]
It follows that the condition that $log(S)$ solves equation ( \ref{eq1}) is equivalent to 
\[ I_1 = -S \]
\[ I_2 = \frac{| \nabla'(S)|^2}{S} \]
\[ I_3 = M S^2. \]

Let $v$ be the vector with components 
 $\left({g'}_{1,\bar{1}}, \ldots {g'}_{n,\bar{n} } \right) $ and $w$ be the vector with components 
$ \left( 1,1,\ldots,1 \right),$  in the proof of the inequality $I_3 \leq M S^2,$ Wu and Yau used the Schwarz inequality $<v,w>^2 \leq ||v||^2||w||^2,$ i.e. the comparison of the $l^1$ and $l^2$ norm of the vector $v.$ Therefore if $I_3 = MS^2$ then $\omega'(x) = \lambda(x) \omega(x),$ with $\lambda(x) = \frac{n}{S(x)}.$ Since the $(1,1)$ forms $\omega'$ and 
$ \frac{n}{S} \omega$ are intrinsic, and for every $x \in X$ there exists a coordinate system centered at $x$ where they coincide,  they must coincide everywhere. Since both $\omega$ and $\omega'$ are closed $(1,1)$ forms, and since $n \geq 2$ we conclude that $\lambda = \frac{n}{S}$ is constant.  Now observe that any unit vector   $\alpha_1$ in $x$ for the metric $\omega(x)$ can be the first vector of an $\omega(x)$ orthonormal basis $\alpha_1, \ldots, \alpha_n.$ Since  
 $ g'(x) = \lambda(x) g(x)$   such basis is also $\omega'(x)$ orthogonal, and we can choose  an holomorphic coordinate system where $\alpha_j =  \frac{\partial }{\partial z_j}$ at $x$ for $1 \leq j \leq n.$ 
Now  if $\Theta$ is the curvature form of $\omega$ at $x$ then by definition $ \Theta( \alpha_j \otimes \alpha_j, \alpha_j \otimes \alpha_j)  - H(x)  \leq 0$ for all $ 1 \leq j \leq n,$ in particular 
\begin{equation}  \sum_{j=1}^n \left(  \Theta( \alpha_j \otimes \alpha_j, \alpha_j \otimes \alpha_j)  - H(x) \right)  \leq 0. \label{Royden} \end{equation} 
However the proof of the inequality $I_3 \leq MS^2$ and of the Royden Lemma shows that if we have $I_3 = MS^2$ then
$ \Theta( \alpha_j \otimes \alpha_j, \alpha_j \otimes \alpha_j) = - H(x) $ for all $1 \leq j \leq n.$ Since $\alpha_1$ is an arbitrary $\omega(x)$ unit vector, we conclude that the holomorphic sectional curvature of $\omega$ at $x$ does not depend on the direction chosen, since $n \geq 2$ the holomorphic sectional curvature is  a negative constant, see \cite{K-N} chapter IX pag. 168 Theorem 7.5. Therefore $X$ is covered by the ball, see  \cite{K-N} chapter IX pag 170 Theorem 7.9.
\end{proof}

\begin{remark}

As we observed above the function $f - \log(\overline{M})$ solves equation ( \ref{eq1}) if and only if the  function $M$ is a positive constant. 
This is the case for example if $X$ is covered by a bounded hermitian symmetric domain  and $\omega = \lambda \omega'$ with $\lambda$ positive constant.  So in general the condition $M$ constant does not imply that $X$ is covered by the ball.  Assume that $M$ is a positive constant, is then $X$ covered by a K\"ahler manifold where the $\omega$ isometry group acts transitively ? Is the metric $\omega$ K\"ahler Einstein in this case ? \end{remark}

We also have:

\begin{theorem} 
Let $(X,\omega)$ be as above then 
the following conditions are equivalent :
\[ \begin{array}{l} 1) \  \omega = \lambda \omega'
\mbox{ for some positive constant $\lambda $,}  
 \\ \mbox{ and the manifold X is covered by the ball}  \\   
2) \mbox{ $log(S) = f - log(\overline{M})$    everywhere in $X$} \\ 
3) \mbox{ $log(S) = f - log(\overline{M}) $ at some point in $X$}   \end{array} \] \label{theo-ineq}
\end{theorem}

\begin{proof}
Assume condition 1) holds, then as we have seen above, the function $M = \overline{M} =  \frac{1}{ n \lambda}$ is constant, hence the function $f$ is identically $0,$  $S = n \lambda$ and $ S \overline{M} \equiv 1,$ therefore condition 2) holds. Viceversa if condition 2) holds, then $log(S) = f - log(\overline{M})$ solves equation (\ref{eq1}), since it is both a sub and a super solution, hence  we obtain condition 1) by Theorem \ref{dim1} i) if $n =1$ and by Theorem \ref{dim1} ii) if $n \geq 2.$   To see that 2) is equivalent to 3) we apply the maximum principle.  Let 
$ \psi = log(S) - f + log(\overline{M}).$

Note that the function $\psi $ satisfies differential inequality  

\begin{equation} 
\Delta_{\omega'} (\psi) \geq L \psi. \label{difference} \end{equation} 

Where $L =  \frac{M}{\overline{M}} exp(f) \left( \frac{exp(\psi) - 1}{\psi} \right)$ (here we define  $ \left( \frac{exp(\psi) - 1}{\psi} \right)(p) = 1$ if $\psi(p) = 0$). By (\ref{ineq4}) we have that   $ \psi \leq 0,$  and that $ L \geq  0.$ Therefore by the strong maximum principle, \cite{G-T} Theorem 3.5,  if $\psi$ vanishes at some point in $X$ it is identically $0.$ \end{proof}

\end{document}